\documentclass[11pt, pdflatex]{amsart}

\usepackage[latin1]{inputenc}
\usepackage[english]{babel}

\usepackage{amsmath, amsthm, amsfonts, amssymb}
\usepackage{graphicx}
\usepackage{caption}
\newcommand{\abs}   [1]{\left \vert #1 \right \vert}

\newcommand{\norm}  [2]{\left \Vert #1 \right \Vert_{#2}}

\newcommand{\bo}    [0]{{\partial \Omega}}

\makeatletter
\newtheorem*{rep@theorem}{\rep@title}
\newcommand{\newreptheorem}[2]{%
\newenvironment{rep#1}[1]{%
 \def\rep@title{#2 \ref{##1}}%
 \begin{rep@theorem}}%
 {\end{rep@theorem}}}
\makeatother

\newtheorem{theorem}{Theorem}[section]

\newtheorem{lemma}[theorem]{Lemma}

\newtheorem{remark}[theorem]{Remark}

\newreptheorem{theorem}{Theorem}

\title[Reconstruction of rough potentials in the plane]{Reconstruction of rough potentials in the plane}
\author{Jorge Tejero}

\begin{document}
\maketitle

\begin{abstract}
We provide a reconstruction scheme for complex-valued potentials in $H^s (\mathbb{R}^2)$ for $s > 0$. The procedure extends the method of Bukhgeim relying on quadratic exponential solutions. We also see how the new reconstruction formulas proposed improve the convergence on a set of examples.
\end{abstract}

\section{Introduction}

Consider the time-independent Schr\"odinger equation with potential $q$ in a bounded domain $\Omega$ in $\mathbb{R}^2$. Given a function $f$ defined on the boundary, the Dirichlet problem for this equation is to determine $u$ satisfying 
\begin{align} \label{dirichletProblem}
\left\lbrace
\begin{array}{l}
\ \, \Delta u = q u \quad \text{in } \Omega \\
u |_\bo = f .
\end{array} \right.
\end{align}
Provided this can be solved uniquely, we can formally define the Dirichlet to Neumann (DtN) map $\Lambda_q$ by
\begin{align*}
\Lambda_q [f] = \frac{\partial u}{\partial \nu}
\end{align*}
where $\nu$ is the outward normal unitary vector to the boundary.

Our goal is to recover the potential inside $\Omega$ from the DtN map. This problem dates back to Gel'fand in~\cite{G54} and is closely related to the inverse conductivity problem proposed by Calder\'on in~\cite{C80}. Some relevant works in higher dimensions are \cite{SU87, N88, HT13, CR16}. 

The two-dimensional case differs significantly to the higher dimensional one, and different techniques are used to address the problem.
In~\cite{N96} Nachman introduced the $\overline{\partial}$-method to prove uniqueness 
for conductivities in $W^{2,p}$ with $p>1$,
and gave a reconstruction procedure. 
In~\cite{BU97} Brown and Uhlmann reduced the conductivity equation to a first order system and proved uniqueness for conductivities with one derivative.
In~\cite{AP06}, combining the $\overline{\partial}$-method with the theory of 
quasi-conformal maps, Astala and P\"{a}iv\"{a}rinta gave a reconstruction procedure for $L^\infty$ conductivities. 

There was little progress for general potentials in the plane until \cite{B08}, where Bukhgeim introduced a new method, proving uniqueness for $C^1$-potentials. In this work
Bukhgeim used solutions of the form
\begin{align} \label{bukhgeimSolutions}
u_{\lambda, x} = e^{i \lambda \psi_x} (1 + w_{\lambda, x}), \quad \quad \psi_x (z) = \tfrac{1}{2}(z_1 - x_1 + i z_2 -  i x_2 )^2 ,
\end{align}
where $w_{\lambda, x}$ is small in some suitable sense. These solutions provide the following (formal) recovery formula for smooth enough potentials
\begin{align} \label{standardRecoveryFormula}
\lim_{\lambda \to \infty} \frac{\lambda}{\pi} \int_{\bo} e^{i \lambda \overline{\psi}_x} (\Lambda_q - \Lambda_0) [u_{\lambda, x}] = q(x).
\end{align}

This quadratic phase approach has since been extended by several authors; we list a few here. In \cite{BIY15}, Bl\r{a}sten, Imanuvilov and Yamamoto proved uniqueness
for potentials in $L^p$ with $p>2$, and gave stability estimates in the $L^2$ norm for potentials in $H^s$ with $s > 0$. In \cite{AFR16}, Astala, Faraco and Rogers gave a reconstruction procedure for potentials in~$H^{1/2}$ and proved that their procedure fails with less regularity. In \cite{LTV17} Lakshtanov, Tejero and Vainberg used the quadratic phase solutions in the first order system introduced in \cite{BU97} to prove uniqueness for complex-valued Lipschitz conductivities.

Using a different approach, Lakshtanov, Novikov and Vainberg \cite{LNV16} also considered general potentials, providing a reconstruction scheme for generic real-valued bounded potentials. Their procedure relies on Faddeev's scattering solutions and allows to recover almost each potential, in the sense of their Remark 4.1. This work has been later extended in \cite{LV16} to deal with generic real-valued potentials in $L^p$, with $p>1$.

In this paper we exploit the fact that the reconstruction formula \eqref{standardRecoveryFormula} commutes, in some sense, with taking averages. This allows us to reduce the problem to the recovery of smoother potentials.

Let $\varphi$ be a Schwartz function, supported in the unit ball, satisfying
\begin{align*}
\varphi (x) \geq 0\quad \text{and}\quad \int_{\mathbb{R}^2} \varphi(x) \, dx = 1, 
\end{align*}
and we write $\varphi_\sigma (x) := \sigma^{-2} \, \varphi(\sigma^{-1} x)$ for $\sigma > 0$. For $x\in \Omega$ we denote the boundary information at frequency $\lambda$ by
\begin{align*}
BI_\lambda (x) := \frac{\lambda}{\pi} \Big \langle (\Lambda_q - \Lambda_0) [u_{\lambda, x} |_\bo ], e^{i \lambda \overline{\psi}_{x}} |_\bo \Big \rangle 
\end{align*}
where $u_{\lambda, x}$ are Bukhgeim solutions of the form \eqref{bukhgeimSolutions}.
We refer to \cite[Theorem 1.1]{AFR16} for how to compute the values of the Bukhgeim solutions on the boundary from the DtN map. 
The following result is proved in Section \ref{proofMain}.
\begin{theorem} \label{mainTheorem}
Let $s > 0$, let $q \in H^s(\mathbb{R}^2)$ be a complex-valued potential supported in a bounded Lipschitz domain $\Omega\subset \mathbb{R}^2$  and let $\sigma = \lambda^{-1/4}$. Then 
\begin{align*}
\lim_{\lambda \to \infty} \, \varphi_\sigma \ast BI_\lambda(x) = q(x)\quad \mathrm{a.e.}\ x\in\Omega.
\end{align*}
\end{theorem}

We also present a different reconstruction formula that also relies on an averaging procedure. In this case we average over rotatations of the boundary inputs and over a range of the frequencies $\lambda$. We are able to prove a result similar to Theorem \ref{mainTheorem} for this formula, but requiring some additional local regularity. This formula is of interest for the purpose of numerical implementation. We perform some numerical computations comparing the convergence of the main term of the standard Bukhgeim formula with that of the averaged formulas. We see that both are useful, and can be combined together. 

In Section \ref{preliminaries} we introduce some definitions and state some preliminary results. In Section \ref{proofMain} we prove Theorem \ref{mainTheorem} and we give a bound for dimension of the set of points where the recovery fails and an estimate for the decay rate of the error of the reconstruction. In Section \ref{polarAveraging} we introduce the reconstruction formula relying on the rotation and frequency average and prove its convergence under some additional hypothesis. In Section \ref{numericalTests} we perform some numerical tests to study how the averaging procedures improve convergence of the reconstruction formula to the potential.

%%%%%%%%%%%%%%%%%%%%%%%%%%%%%%%%%%%%%%%%%%%%%%%%%%%%%%%%%%%%%%%%%%%%%%%%%%%%%%
%%%%%%%%%%%%%%%%%%%%%%%%%%%%%%%%%%%%%%%%%%%%%%%%%%%%%%%%%%%%%%%%%%%%%%%%%%%%%%
\section{Preliminaries} \label{preliminaries}

Suppose that $0$ is not a Dirichlet eigenvalue for the Hamiltonian $q - \Delta$. Then, for each $f \in H^{1/2} (\bo)$ there exists a unique solution $u \in H^1 (\Omega)$ to the dirichlet problem \eqref{dirichletProblem}, and the DtN map can be defined by 
\begin{align*}
\langle \Lambda_q[f], v\vert_{\partial \Omega} \rangle := \int_\Omega q \, u \, v + \nabla u \cdot \nabla v 
\end{align*}
for any $v \in H^1(\Omega)$. By integration by parts, this definition coincides with that of the introduction whenever the boundary and solutions have sufficient regularity. Taking $v$ harmonic we obtain Alessandrini's identity
\begin{align} \label{ale}
\langle (\Lambda_q - \Lambda_0)[f], v\vert_{\partial \Omega} \rangle = \int_\Omega q \, u \, v.
\end{align}

Let $\phi_x(z) := \psi_x(z) + \overline{\psi_x(z)} = (z_1 - x_1)^2 - (z_2 - x_2)^2$ and let 
\begin{align} \label{tOperator}
T^\lambda_g [f] (x) := \frac{\lambda}{\pi} \int_\Omega e^{i \lambda \phi_x (z)} \, f(z) \, g_{\lambda, x}(z) \, dz.
\end{align}
We also use the notation $T^\lambda [f] := T^\lambda_1 [f]$ for simplicity.
The following two results are due to Astala, Faraco and Rogers, and their proof can be found in \cite{AFR16}. The first result allows to bound the remainder term in the reconstruction, while the second one is a stationary phase type statement, suitably adapted to our purposes; we include the proof for the second as the original result is stated in terms of inhomogeneus Sobolev spaces.

\begin{theorem} \cite[Theorem 4.1]{AFR16} \label{remainderSize}
Let $q \in \dot{H}^s (\mathbb{R}^2)$ with $0 < s < 1$. Then 
\begin{align*}
\lim_{\lambda \to \infty} T^\lambda_w [q] (x) = 0, \quad x \in \Omega.
\end{align*}
Moreover, if $\lambda \geq \left( 1 + c\norm{q}{\dot{H}^s} \right)^{\max \lbrace \frac{1}{2s}, \frac{1}{1-s}\rbrace}$, then
\begin{align*}
\sup_{x \in \Omega} |T^\lambda_w [q](x)| \leq C \, \lambda^{-s} \, \norm{q}{\dot{H}^s}^2.
\end{align*}
\end{theorem}

\begin{lemma} \cite[Lemma 4.2]{AFR16} \label{statPhaseHs}
Let $q \in H^s (\mathbb{R}^2)$ with $1<s<3$. Then
\begin{align*}
\abs{ T^\lambda[q](x) - q(x)} \leq C \lambda^{\frac{1-s}{2}} \norm{q}{\dot{H}^s}, \quad x \in \Omega. 
\end{align*}
\end{lemma}

\begin{proof}
By the Fourier inversion formula and the Cauchy-Schwarz inequality,
\begin{align*}
\left|T^\lambda [q] (x)  - q(x)\right|
&= \bigg| \frac{\lambda}{\pi} \bigg( q \ast e^{i \lambda (z_1^2 - z_2^2)} \Big) (x) - q(x) \bigg| \\
&= \frac{1}{4\pi^2} \bigg| \int e^{ix \cdot \xi} \, \widehat{q}(\xi) \, \Big( e^{-i \frac{1}{\lambda} (\xi_1^2 - \xi_2^2)} - 1 \Big) \, d\xi \bigg| \\
&\leq \norm{q}{\dot{H}^{s}} \bigg( \int \frac{\big| e^{-i \frac{1}{\lambda} (\xi_1^2 - \xi_2^2)} - 1 \big|^2}{|\xi|^{2s}} \, d\xi \bigg)^{1/2} \\
&= \lambda^{\frac{1-s}{2}} \norm{q}{\dot{H}^{s}} \bigg( \int \frac{2 - 2 \cos (\xi_1^2 - \xi_2^2)}{|\xi|^{2s}} \, d\xi \bigg)^{1/2} \\
&= 2 \lambda^{\frac{1-s}{2}} \norm{q}{\dot{H}^{s}} \bigg( \int \frac{\sin^2 \big( \tfrac{1}{2}(\xi_1^2 - \xi_2^2) \big)}{|\xi|^{2s}} \, d\xi \bigg)^{1/2} \\
&\leq 2 \lambda^{\frac{1-s}{2}} \norm{q}{\dot{H}^{s}} \bigg( \int_{\mathbb{D}} \frac{1}{|\xi|^{2(s-2)}} \, d\xi + \int_{\mathbb{R}^2 \setminus \mathbb{D}} \frac{1}{|\xi|^{2s}} \, d\xi \bigg)^{1/2},
\end{align*}
where we have used the trigonometric identity $2 \sin^2 \theta = 1 - \cos 2 \theta$ and the fact that $\sin \theta \leq |\theta|$.
\end{proof}

Now we introduce two lemmas for controlling the norm and the rate of convergence of the mollified potential.

\begin{lemma} \label{convolutionSize}
Let $q \in \dot{H}^s (\mathbb{R}^2)$ with $0 < s < s'$. Then 
\begin{align*}
\norm{\varphi_\sigma \ast q}{\dot{H}^{s'}} \leq C \, \sigma^{s - s'} \, \norm{q}{\dot{H}^s} .
\end{align*}
\end{lemma}

\begin{proof}
Given that $\widehat{\varphi_\sigma} (\xi) = \widehat{\varphi} (\sigma \xi)$, 
by H\"older's inequality we have
\begin{align*}
\norm{\varphi_\sigma \ast q}{\dot{H}^{s'}} 
&\leq \norm{q}{\dot{H}^s} \, \sup_{\xi} \left\lbrace |\xi|^{s' - s} \, \widehat{\varphi} (\sigma \xi) \right\rbrace \\
&= \sigma^{s - s'} \, \norm{q}{\dot{H}^s} \, \sup_{\zeta} \left\lbrace |\zeta|^{s' - s} \, \widehat{\varphi} (\zeta) \right\rbrace.
\end{align*}
As $\varphi$ belongs to the Schwartz space, the proof is concluded.
\end{proof}

\begin{lemma} \label{convolutionConvergence}
Let $q \in C^{0,\alpha} (\mathbb{R}^2)$. Then 
\begin{align*}
\norm{\varphi_\sigma \ast q - q}{L^\infty} \leq \sigma^\alpha \, |q|_{C^{0,\alpha}}. 
\end{align*}
\end{lemma}

\begin{proof}
For $x \in \mathbb{R}^2$ we have
\begin{align*}
| \varphi_\sigma \ast q (x) - q(x) |
&=  \left| q(x) - \int_{B(x,\sigma)} q(y) \, \varphi_\sigma (x-y) \, dy  \right| \\
&\leq  \int_{B(x,\sigma)} \left| q(x) - q(y) \right| \varphi_\sigma (x-y) \, dy .
\end{align*}
Then by H\"older's inequality,
\begin{align*}
| \varphi_\sigma \ast q (x) - q(x) |
&\leq  \int_{B(x,\sigma)} \frac{\left| q(x) - q(y) \right|}{|x - y|^\alpha} |x - y|^\alpha \varphi_\sigma (x-y) \, dy \\
&\leq |q|_{C^{0,\alpha}} \int_{B(x,\sigma)} |x-y|^\alpha \varphi_\sigma (x-y) \, dy \\
&\leq \sigma^\alpha \, |q|_{C^{0,\alpha}},
\end{align*}
which completes the proof.
\end{proof}

%%%%%%%%%%%%%%%%%%%%%%%%%%%%%%%%%%%%%%%%%%%%%%%%%%%%%%%%%%%%%%%%%%%%%%%%%%%%%%
%%%%%%%%%%%%%%%%%%%%%%%%%%%%%%%%%%%%%%%%%%%%%%%%%%%%%%%%%%%%%%%%%%%%%%%%%%%%%%
\section{Proof of Theorem \ref{mainTheorem}} \label{proofMain}

We use $\text{dim}_H \lbrace E \rbrace $ to denote the Hausdorff dimension of $E$. Theorem \ref{mainTheorem} is a consequence of the following more precise statement.
\begin{theorem} 
Let $s > 0$, let $q \in H^s(\mathbb{R}^2)$ be a complex-valued potential supported in a bounded Lipschitz domain $\Omega\subset \mathbb{R}^2$  and let $\sigma = \lambda^{-1/4}$. Then 
\begin{align*}
\mathrm{dim}_H \left\lbrace x: \lim_{\lambda \to \infty} \, \varphi_\sigma \ast BI_\lambda (x) \neq q(x) \right\rbrace \leq 2 - 2 s.
\end{align*}
\end{theorem}

\begin{proof}
Let $T^\lambda_g$ as in \eqref{tOperator}. By Alessandrini's identity \eqref{ale} we have
\begin{align*}
BI_\lambda (x) &= \frac{\lambda}{\pi} \int_\Omega e^{i \lambda \phi_x (z)} \, q(z) (1 + w_{\lambda, x}(z)) \, dz = T^{\lambda} [q] (x) + T^\lambda_w [q] (x)
\end{align*}
leading to 
\begin{align*} 
\varphi_\sigma \ast BI_\lambda  =  \varphi_\sigma \ast T^{\lambda} [q]  + \varphi_\sigma \ast T^\lambda_w [q].
\end{align*}
Let $q_\sigma := \varphi_\sigma \ast q$. Using the triangle inequality we have
\begin{align} \label{integralSplit}
| \varphi_\sigma \ast BI_\lambda - q | &\leq | \varphi_\sigma \ast T^{\lambda}[q] - q_\sigma | + | q_\sigma  - q | + | \varphi_\sigma \ast T^\lambda_w[q] |. 
\end{align}

For the first term, we can use Fubini's theorem to obtain
\begin{align*}
\varphi_\sigma \ast T^{\lambda} [q] (x) 
&= \frac{\lambda}{\pi} \int_{\mathbb{R}^2} \varphi_\sigma (y) \int_\Omega e^{i \lambda \phi_{x-y} (z)} \, q(z) \, dz \, dy \\
&= \frac{\lambda}{\pi} \int_{\mathbb{R}^2} \varphi_\sigma (y) \int_{\Omega} e^{i \lambda \phi_{x} (z+y)} \, q(z) \, dz \, dy
\\
&= \frac{\lambda}{\pi} \int_{\mathbb{R}^2} e^{i \lambda \phi_{x} (z)} \int_{\mathbb{R}^2} \varphi_\sigma (y) q(z-y) \, dy \, dz
\end{align*}
leading to
\begin{align*}
\varphi_\sigma \ast T^{\lambda} [q] =  T^{\lambda} [q_\sigma].
\end{align*}
Using Lemma \ref{statPhaseHs} and Lemma \ref{convolutionSize}, for $s'$ satisfying $1<s'<3$, we have
\begin{align*} %\label{statOnConvolution}
\abs{ T^\lambda[q_\sigma] - q_\sigma} \leq C \, \lambda^{\frac{1-s'}{2}} \norm{q_\sigma}{\dot{H}^{s'}} \leq C \, \lambda^{\frac{1-s'}{2}} \sigma^{s - s'} \, \norm{q}{\dot{H}^s}.
\end{align*}
To deal with the third term, we note that by Lemma \ref{remainderSize} we have
$$
 | \varphi_\sigma \ast T^\lambda_w[q] |\le \norm{T^\lambda_{w}[q]}{L^\infty} \leq C \lambda^{-s} \norm{q}{\dot{H}^{s}}^2. 
$$
Plugging these estimates into \eqref{integralSplit},  we can see that
\begin{align} \label{errorTerm}
| \varphi_\sigma \ast BI_\lambda - q| &\leq C \, \lambda^{\frac{1-s'}{2}} \sigma^{s - s'} \, \norm{q}{\dot{H}^s} + C \, \lambda^{-s} \norm{q}{\dot{H}^s}^2  + | q_\sigma- q|. 
\end{align}
Thus we have 
\begin{align*} 
| \varphi_\sigma \ast BI_\lambda (x) - q(x)| &\leq C \, \lambda^{-\kappa} ( \norm{q}{\dot{H}^s} +  \norm{q}{\dot{H}^s}^2 ) + | q_\sigma (x)- q(x)|
\end{align*}
for any $\kappa$ satisfying $0<\kappa<\min \lbrace (1+s)/4, s \rbrace$. 
It only remains to see that
\begin{align*}
\text{dim}_H \left\lbrace x : \lim_{\sigma \to 0} \, q_\sigma (x) \neq q(x) \right\rbrace \leq 2 - 2 s
\end{align*}
which follows from, for example, \cite[Lemma A.1]{BBCR11}.
\end{proof}

\begin{remark}
If the potential satisfies the H\"older condition in some neighborhood of the point to reconstruct, we can obtain an estimation for the decay rate of the error term. More precisely, for $x$ such that $|q|_{C^{0,\alpha}(B(x,r))} < \infty$ for some $r,\alpha>0$, we have
\begin{align*}
| \varphi_\sigma \ast BI_\lambda (x) - q(x) | \leq C \, \lambda^{- \min \lbrace s, \alpha \rbrace / 4} \left( \norm{q}{\dot{H}^s} + \norm{q}{\dot{H}^s}^2 + |q|_{C^{0,\alpha}(B(x,r))} \right).
\end{align*}
This follows from using Lemma \ref{convolutionConvergence} to bound $| q_\sigma (x) - q (x) |$ in \eqref{errorTerm}.
\end{remark}

%%%%%%%%%%%%%%%%%%%%%%%%%%%%%%%%%%%%%%%%%%%%%%%%%%%%%%%%%%%%%%%%%%%%%%%%%%%%%%
%%%%%%%%%%%%%%%%%%%%%%%%%%%%%%%%%%%%%%%%%%%%%%%%%%%%%%%%%%%%%%%%%%%%%%%%%%%%%%
\section{Polar averaging} \label{polarAveraging}

First we introduce some averaging operators which are used to express our recovery formula involving \textit{polar averaging}. 
For $f:\mathbb{R}^+ \times [0, 2\pi) \to \mathbb{C}$ we define the {\it angular averaging} operator by
\begin{align*}
A_{\mathrm{ang}}[f](r,\theta) := \frac{1}{2\pi} \int_0^{2\pi} f (r, \alpha) \, d\alpha ,
\end{align*}
and the {\it radial smoothing} operator by
\begin{align*}
S_{\mathrm{rad}}[f](r,\theta) := \int_0^1 f \big( r (1+s)^{-1/2} ,\theta\big) \, ds .
\end{align*}
For $\lambda \in \mathbb{R}^+$ and $F\in L^1_{loc}(\mathbb{R}^+ )$ we define the {\it frequency averaging} operator
\begin{align*} 
A_{\mathrm{freq}} [F](\lambda) := \frac{1}{\lambda}\int_\lambda^{2\lambda} F(t)\,dt. 
\end{align*}
In the following lemma, we connect the frequency averaging with the radial smoothing operator.
\begin{lemma} \label{ALambdaToAS}
Let $f \in L^1(\mathbb{R}^2)$. Then
\begin{align*} 
A_{\mathrm{freq}}\big[ T^{(\cdot)} [f](0) \big](\lambda) = T^{\lambda} \big[S_{\mathrm{rad}}[f]\big](0).
\end{align*}
\end{lemma}

\begin{proof}
First we note that by a change of variables
\begin{align*} 
A_{\mathrm{freq}} [F](\lambda) = \frac{1}{\lambda}\int_\lambda^{2\lambda} F(t)\,dt= \int_0^1 F(\lambda (1 + s) ) \, ds
\end{align*}
so that $A_{\mathrm{freq}}\big[ T^{(\cdot)} [f](0) \big](\lambda)$ can be written as 
\begin{align*}
\int_0^1  T^{\lambda (1 + s)} [f](0) \, ds 
= \int_0^1 \frac{\lambda (1 + s)}{\pi} \int_{\mathbb{R}^2} e^{i \lambda(1 + s) (z_1^2 - z_2^2)} f(z) \, dz \, ds .
\end{align*}
Now, switching to polar coordinates, changing variables $r = (1+s)^{1/2} \, \rho$ and applying Fubini's theorem,  we find
\begin{align*}
& \quad \ \ \int_0^1 \frac{\lambda (1 + s)}{\pi} \int_{\mathbb{R}^2} e^{i \lambda (1 + s)(z_1^2 - z_2^2)} f(z) \, dz \, ds \\
&= \int_0^1 \frac{\lambda (1 + s)}{\pi}  \int_0^\infty \int_0^{2\pi} e^{i \lambda (1 + s) \rho^2 \cos (2 \theta)}  \, f (\rho, \theta) \,  d\theta \,\rho d\rho  \, ds \\
&= \frac{\lambda }{\pi} \int_0^1 \int_0^\infty \int_0^{2\pi} f \left( r \, (1 + s)^{-1/2}, \theta \right) \,  e^{i \lambda r^2 \cos (2 \theta)}  \, d\theta \, rdr  \, ds \\
&= \frac{\lambda }{\pi} \int_0^\infty \int_0^{2\pi} \int_0^1 f \left( r \, (1 + s)^{-1/2}, \theta \right) \, ds  \, e^{i \lambda r^2 \cos (2 \theta)}  \, d\theta \,r dr \\
&= \frac{\lambda }{\pi}  \int_0^\infty \int_0^{2\pi} S_{\mathrm{rad}}[f](r, \theta)\, e^{i \lambda r^2 \cos (2 \theta)} \, d\theta \, rdr \\
&= T^{\lambda} \big[S_{\mathrm{rad}}[f]\big](0)
\end{align*}
and the proof is concluded.
\end{proof}

We use the notation $C^k_0$ with $k \in \mathbb{N}$ to refer to the space of continuous functions with $k$ continuous derivatives and compact support.
The following lemma is useful to see that the angular averaging operator  preserves regularity.
\begin{lemma} \label{rotationalAverageProperty}
If $f \in C^k_0 (\mathbb{R}^2)$, then  $A_{\mathrm{ang}}[f]\in C^k_0 (\mathbb{R}^2)$.
\end{lemma}

\begin{proof}
For the case $k=0$, note that $f$ is uniformly continuous given that it has compact support. Thus, for any $\epsilon>0$ there exists some $\delta > 0$ such that, for $x,y \in \mathbb{R}^2$ satisfying $|x-y| < \delta$ we have $|f(x)-f(y)| < \epsilon$, so that
\begin{align*}
|A_{\mathrm{ang}}[f](x) - A_{\mathrm{ang}}[f](y) | = \left| \frac{1}{2 \pi} \int_0^{2\pi} f(|x|, \theta) - f(|y|, \theta) \, d\theta \right| < \epsilon
\end{align*} 
whenever $|x-y| < \delta$.

For $k>0$, as $A_{\mathrm{ang}}[f]$ is radial,  the angular derivatives are zero. For the radial derivative,  we note that
\begin{align*}
\frac{\partial^k }{\partial r^k}A_{\mathrm{ang}}[f](r, \theta) = \frac{1}{2 \pi} \int_0^{2\pi} \frac{\partial^k f}{\partial r^k}(r, \alpha) \, d\alpha  .
\end{align*}
Applying the previous argument to the derivatives of the function concludes the proof.
\end{proof}

The following lemma describes how $S_{\mathrm{rad}}$, considered here as acting on one-dimensional functions,
\begin{align*}
S_{\mathrm{rad}}[f](r) := \int_0^1 f \big( r (1+s)^{-1/2}\big) \, ds ,
\end{align*}
regularizes away from the origin while at the same time it preserves regularity in the whole domain.
\begin{lemma} \label{frequencyAveraging}
Let $f \in L^1(\mathbb{R}^+)$. Then

(i) $S_{\mathrm{rad}}[f] \in C^0(0, \infty)$.

(ii) If $f \in C^k(0, \infty)$, then $S_{\mathrm{rad}}[f] \in C^{k+1}(0, \infty)$.

(iii) If $f \in C^k[0, \infty)$, then $S_{\mathrm{rad}}[f] \in C^{k}[0, \infty)$.

(iv) If $\text{supp} (f) \subset (a, b)$, then $\text{supp} (S_{\mathrm{rad}}[f]) \subset (a , b \sqrt{2})$.
\end{lemma}

\begin{proof}
Let $\epsilon > 0$ and write $g=S_{\mathrm{rad}}[f]$. Then 
\begin{align*} 
g(t + \epsilon) - g(t) = \int_0^1 f \big( (t + \epsilon) (1 + s)^{-1/2} \big) \, ds - \int_0^1 f \big( t \, (1 + s)^{-1/2} \big) \, ds.
\end{align*}
Changing variables in the first integral by taking  
\begin{align} \label{changeOfVariables1}
s' = \frac{t + \epsilon}{t}(1 + s)^{-1/2}
\end{align} 
and in the second integral taking
\begin{align} \label{changeOfVariables2}
s' = (1 + s)^{-1/2}
\end{align}
we have
\begin{align*} 
g(t + \epsilon) - g(t)  
&= 2 \int_{\frac{1}{\sqrt{2}} + \frac{\epsilon}{t\sqrt{2}}}^{1 + \frac{\epsilon}{t}} f(t \, s) \, \left(\frac{t+\epsilon}{t} \right)^2 \, s^{-3} \, ds -
2 \int_{\frac{1}{\sqrt{2}}}^1 f(t \, s)  \, s^{-3} \, ds \nonumber \\ 
&=  2 \left(\frac{t+\epsilon}{t} \right)^2 \int_{1}^{1 + \frac{\epsilon}{t}} f(t \, s) \, s^{-3} \, ds -
2 \int_{\frac{1}{\sqrt{2}}}^{\frac{1}{\sqrt{2}} + \frac{\epsilon}{t\sqrt{2}}} f(t \, s) \, s^{-3} \, ds \\ 
& \quad +2 \frac{2 \, t \, \epsilon + \epsilon^2}{t^2}  \int_{\frac{1}{\sqrt{2}} + \frac{\epsilon}{t\sqrt{2}}}^{1} f(t \, s) \, s^{-3} \, ds. \nonumber
\end{align*}
Given some fixed $t>0$ we have $f(t \, s) \, s^{-3} \in L^1((1/\sqrt{2},\infty))$. By absolute continuity with respect to the Lebesgue measure, the first two integrals are smaller than $\delta$ for any $\delta>0$ by taking $\epsilon$ small enough. The third term is smaller than $6 \, \epsilon \, t^{-1} \| f(t \, s) \, s^{-3} \|_{L^1((1/\sqrt{2},\infty))}$. Therefore we know that $g$ is right-continuous in $(0,\infty)$. Taking $\epsilon < t$ a similar change of variables yields
\begin{align*}
g(t) - g(t-\epsilon) &= 2 \int_{1 - \frac{\epsilon}{t}}^1 f(t \, s) \, s^{-3} \, ds - 2 \left( \frac{t-\epsilon}{t} \right)^2 \int_{\frac{1}{\sqrt{2}} - \frac{\epsilon}{t\sqrt{2}}}^{\frac{1}{\sqrt{2}}} f(t \, s) \, s^{-3} \, ds \\
& \quad + 2 \left( \frac{2t\epsilon - \epsilon^2}{t^2} \right) \int_{\frac{1}{\sqrt{2}}}^{1 - \frac{\epsilon}{t}} f(t \, s) \, s^{-3} \, ds
\end{align*}
and by the same arguments we see that $g$ is left-continuous in $(0,\infty)$, concluding the proof of \textit{(i)}.

Now, differentiating under the sign of the integral we obtain
\begin{align*}
g^{(k)}(t) = \int_0^1 (1+s)^{-k/2} \, f^{(k)} \big( t \, (1+s)^{-1/2} \big) \, ds 
\end{align*}
and using the same change of variables as in \eqref{changeOfVariables1} and \eqref{changeOfVariables2} we obtain
\begin{align*}
& \quad \ \ g^{(k)}(t + \epsilon ) - g^{(k)}(t) \\
&= \frac{2 \, (t+ \epsilon )^{2-k}}{t^{2-k}} \int^{1 + \frac{\epsilon}{t}}_{\frac{1}{\sqrt{2}} + \frac{\epsilon}{t\sqrt{2}}} f^{(k)}(t \, s) \, s^{k-3} \, ds - 2 \int^1_{\sqrt{2}} f^{(k)}(t \, s) \, s^{k-3} \, ds \\
&= \frac{2 \, (t+ \epsilon )^{2-k}}{t^{2-k}} \int_{1}^{1 + \frac{\epsilon}{t}} f^{(k)}(t \, s) \, s^{k-3} \, ds -
2 \int_{\frac{1}{\sqrt{2}}}^{\frac{1}{\sqrt{2}} + \frac{\epsilon}{t\sqrt{2}}} f^{(k)}(t \, s) \, s^{k-3} \, ds \\ 
& \quad + 2 \frac{ (t+ \epsilon )^{2-k} - t^{2-k}}{t^{2-k}}  \int_{\frac{1}{\sqrt{2}} + \frac{\epsilon}{t\sqrt{2}}}^{1} f^{(k)}(t \, s) \, s^{k-3} \, ds .
\end{align*}
For $t>0$ we can divide by $\epsilon$ and take the limit as $\epsilon$ tends to $0$ to get
\begin{align*}
g^{(k+1)}(t) = \frac{2}{t}  f^{(k)}(t) - \frac{2^{2 - k/2}}{t} f^{(k)}(t / \sqrt{2}) + \frac{4 - 2k}{t} \int_{\frac{1}{\sqrt{2}}}^{1} f^{(k)}(ts) \, s^{k-3} \, ds .
\end{align*}
As the right-hand side is continuous in $(0, \infty)$, so is $g^{(k+1)}$, and \textit{(ii)} is proved. 

To prove \textit{(iii)} just see that for any $\delta > 0$ we can take $\epsilon > 0$ such that for  $t < \epsilon$ then $|f^{(k)} (t) - f^{(k)} (0)| < \delta$. Thus
\begin{align*}
\left| g^{(k)}(\epsilon) - g^{(k)}(0) \right| 
& \leq \int_0^1 (1+s)^{-k/2} \left| f^{(k)} \big(\epsilon \, (1+s)^{-1/2} \big) - f^{(k)}(0) \right| \, ds \\
& < \int_0^1 (1+s)^{-k/2} \, \delta \, ds \leq \delta .
\end{align*}
The continuity of $g^{(k)}$ away from zero follows from \textit{(ii)}.

Point \textit{(iv)} is a straightforward consequence of the definition of $S_{\mathrm{rad}}$.
\end{proof}

\begin{remark}
The gain of regularity in point \textit{(ii)} cannot be extended to $[0, \infty)$. We can see this by taking $f(t) = t^{1/2}$.
\end{remark}

We now consider Bukhgeim solutions for the Schr\"odinger equation where the potential has been rotated. Using these solutions, we construct a new family of solutions $u_{\lambda, x, \theta}$ for the Schr\"odinger equation with the original  potential (before taking the rotation). These new solutions depend on an additional parameter $\theta \in [0, 2\pi)$.

More precisely, let $x, z \in \mathbb{R}^2$ and let $R_{x,\theta}(z)$ denote the rotation of $z$ around $x$ by an angle $\theta$, given by
\begin{align*}
R_{x,\theta}(z) = x + \left(
\begin{array}{c c} 
\cos \theta & - \sin \theta \\
\sin \theta & \cos \theta
\end{array}
\right) (z-x).
\end{align*}
Letting $q_{x,\theta} (z) := q \circ R_{x,\theta} (z)$, we consider Bukhgeim solutions  
$$u_{\lambda, x}=e^{i \lambda \psi_x}  (1 + w_{\lambda,x})$$  to the Schr\"odinger equation with rotated potential $\Delta u = q_{x,\theta} \, u$. We write $$u_{\lambda, x, \theta} (z) := u_{\lambda, x} \circ R_{x,-\theta}(z),$$
and consider the boundary information at frequency $\lambda$ defined by
\begin{align*}
BI_x(\lambda) := \frac{\lambda}{2\pi^2} \int_0^{2\pi}\Big \langle (\Lambda_q - \Lambda_0) [u_{\lambda, x, \theta} |_\bo ], e^{i \lambda \overline{\psi}_x \circ R_{x,-\theta}} |_\bo \Big \rangle\, d\theta.
\end{align*}

By the conformal invariance of the Laplacian,
\begin{align*}
\Delta u_{\lambda, x, \theta} = \Delta \big( u_{\lambda, x} \circ R_{x,-\theta} \big) = \big( \Delta u_{\lambda, x} \big) \circ R_{x,-\theta},
\end{align*}
we see that the $u_{\lambda, x, \theta}$ solves the original Schr\"odinger equation $\Delta u = q \, u$. By the same rotational invariance,  $e^{i \lambda \psi_x \circ R_{x,-\theta}}$ is a solution to Laplace's equation, and so we can still use Alessandrini's identity \eqref{ale} to reinterprete the boundary information. Moreover   $u_{\lambda, x, \theta} |_\bo$ can be recovered from $\Lambda_q-\Lambda_0$ by a suitably rotated version of \cite[Theorem 1.1]{AFR16}.

\begin{theorem} \label{rotationalRecoveryTheorem}
Let $s > 0$, let $q \in H^s (\mathbb{R}^2)$ be a complex-valued potential and let $\Omega$ be a bounded Lipschitz domain in the plane. Then, for any $x$ such that $q \in C^2(B(x,r))$ for some $r>0$, we have 
\begin{align*}
\lim_{\lambda \to \infty} A_{\mathrm{freq}}^{\,3}[BI_x](\lambda) = q(x).
\end{align*}
\end{theorem}

\begin{proof}
Without loss of generality, we can suppose that we are recovering the potential at the origin, and so we omit the dependence on $x=0$.
The first step is to split the reconstruction formula into a main term and a remainder term using Alessandrini's identity \eqref{ale}, to obtain
\begin{align*}
BI_{x}(\lambda) 
&= \frac{\lambda}{2\pi^2} \int_0^{2 \pi} \Big \langle (\Lambda_q - \Lambda_0) [u_{\lambda, x, \theta} (z) ], e^{i \lambda \overline{\psi} \circ R_{-\theta}} |_\bo \Big \rangle \, d\theta \\
& = \frac{\lambda}{2\pi^2} \int_0^{2 \pi} \int_{\mathbb{R}^2} q(z)\, e^{i \lambda \phi \circ R_{-\theta}(z)}   (1 + w_{\lambda} \circ R_{-\theta}(z)) \, dz \, d\theta \\
& = \frac{\lambda}{2\pi^2} \int_0^{2 \pi} \int_{\mathbb{R}^2} q \circ R_{\theta}(z) \, e^{i \lambda \phi(z)}  (1 + w_{\lambda}(z)) \, dz \, d\theta \\
& = \frac{1}{2 \pi} \int_0^{2 \pi} T^{\lambda} [ q \circ R_{ \theta} ] + T^\lambda_{w} [ q \circ R_{ \theta}  ]\, d\theta .
\end{align*}
Now, by Fubini's theorem we see that 
\begin{align*}
\frac{1}{2 \pi} \int_0^{2 \pi} T^{\lambda} [ q \circ R_{ \theta} ]\,d\theta
&= \frac{\lambda}{2 \pi^2} \int_{\mathbb{R}^2} e^{i \lambda \phi(z)} \int_0^{2 \pi} \, q \circ R_\theta (z) \, d\theta \, dz \\
&= T^{\lambda} [V_0] 
\end{align*}
where $V_0(z)=A_{\mathrm{ang}} [q](|z|)$.
Thus, we find
\begin{align*}
A_{\mathrm{freq}}^{\,3} [BI_{x}] = A_{\mathrm{freq}}^{\,3} \big[ T^{(\cdot)} [V_0] \big] + A_{\mathrm{freq}}^{\,3} \bigg[ \frac{1}{2\pi} \int_0^{2\pi} T^{(\cdot)}_{w} [ q \circ R_{\theta}]\,d\theta \bigg].
\end{align*}
Now by Lemma~\ref{ALambdaToAS}, we have that
$$
A_{\mathrm{freq}}^{\,3} \big[ T^{(\cdot)} [V_0] \big](\lambda) = T^{\lambda} [V_3], \quad \text{where}\ V_j=S_{\mathrm{rad}}[V_{j-1}], 
$$
and by Lemma \ref{remainderSize} we know that the remainder satisfies
\begin{align*}
\lim_{\lambda \to \infty} T^\lambda_{w} [ q \circ R_{\theta} ] = 0. 
\end{align*}
Thus we find that
\begin{align*}
\lim_{\lambda\to\infty} A_{\mathrm{freq}}^{\,3} [BI_{x}](\lambda) = \lim_{\lambda\to\infty} T^{\lambda} [V_3].
\end{align*}
Noting that $V_3(0)=q(0)$, it remains to prove that $V_3\in H^2 (\mathbb{R}^2)$ so that we can conclude using Lemma~\ref{statPhaseHs}.

To see that $V_3\in H^2 (\mathbb{R}^2)$, recall that $q \in H^s (\mathbb{R}^2)$ is compactly supported, so that $q \in L^1(\mathbb{R}^2)$, which gives us
\begin{align*}
&\int_0^1 | V_0(\rho, \theta)| \, \rho \, d\rho = \int_0^1 \left| \frac{1}{2\pi} \int_0^{2\pi} q(\rho, \theta) \, d\theta \right| \rho \, d\rho \\
& \quad \quad \quad \quad \leq \int_0^1 \int_0^{2\pi} |q(\rho, \theta)|  \, d\theta \, \rho d\rho  < \infty 
\end{align*}
for any $\theta \in [0, 2\pi)$. By Lemma \ref{rotationalAverageProperty} we know that $V_0 \in C^2(B_r)$ so $V_0(\rho, \alpha)$ is bounded for $\rho<r/2$; therefore the one variable function $V_0(\rho)$ belongs to $L^1(\mathbb{R}^+)$.
Let $\varphi \in C^\infty_0 (B_r)$ be a radial function such that $\varphi (z) = 1$ for $|z| < r/2$. As $V_0 \in C^2(B_r)$ we can use part \textit{(iii)} of Lemma \ref{frequencyAveraging}  to obtain
\begin{align*}
S_{\mathrm{rad}} [\varphi \, V_0] \in C^2_0(\mathbb{R}^2), 
\end{align*}
part \textit{(i)} of Lemma \ref{frequencyAveraging}  to gain regularity away from zero
\begin{align*}
S_{\mathrm{rad}} [(1 - \varphi) V_0] \in C^0_0 (\mathbb{R}^2),
\end{align*}
and part \textit{(iv)} of Lemma \ref{frequencyAveraging}  to control the support
\begin{align*}
S_{\mathrm{rad}} [(1 - \varphi) V_0 (\cdot, \alpha)](\rho) = 0, \quad \text{for } \rho < r/2 ,
\end{align*}
leading to
\begin{align*}
V_1 \in C^0_0 (\mathbb{R}^2) \cap C^2(B_{r/2})
\end{align*}
given that $V_1 = S_{\mathrm{rad}} [ \varphi V_0 ] + S_{\mathrm{rad}} [ (1 - \varphi) V_0 ]$.
Using the same arguments (but using part \textit{(ii)} of Lemma \ref{frequencyAveraging}  to gain regularity instead of part \textit{(i)} of Lemma~\ref{frequencyAveraging}) we can see that
$V_2 \in C^1_0 (\mathbb{R}^2) \cap C^2(B_{r/4})$ and finally $V_3 \in C^2_0 (\mathbb{R}^2)\subset H^2(\mathbb{R}^2)$. 
\end{proof}

%%%%%%%%%%%%%%%%%%%%%%%%%%%%%%%%%%%%%%%%%%%%%%%%%%%%%%%%%%%%%%%%%%%%%%%%%%%%%%
%%%%%%%%%%%%%%%%%%%%%%%%%%%%%%%%%%%%%%%%%%%%%%%%%%%%%%%%%%%%%%%%%%%%%%%%%%%%%%
\section{Numerical tests} \label{numericalTests}

The standard reconstruction formula \eqref{standardRecoveryFormula} rests on the decomposition 
\begin{align*}
\frac{\lambda}{\pi} \int_{\bo} e^{i \lambda \overline{\psi}_x} (\Lambda_q - \Lambda_0) [u_{\lambda, x}] = T^\lambda [q] + T^\lambda_w [q]
\end{align*}
where $T^\lambda [q]$ is the main term, an integral defined in the domain which converges to the potential as $\lambda$ grows, and $T^\lambda_w [q]$ is a remainder, which tends to zero.
In this section we perform numerical tests to see how the averaging procedures improve the convergence of the main term to the potential.

We compute the integrals corresponding to the main term by brute force, evaluating the integrand over a sufficiently dense mesh. We use a uniform rectangular mesh, where the number of points varies between 640,000 and 64,000,000 depending on the value of $\lambda$. 
The main term has been computed on a coarser mesh, also regular, with 40,000 points. That is, the resolution of the images shown is of 40,000 pixels.
We do not provide an a priori method for determining the value for the parameter $\sigma$ in the mollifier average; in the results shown the value used is the one that best reduces the $L^1$ error in each example.

The experiments show a significant improvement for the mollifier average and for the angular average, both in terms of the visual image obtained and in terms of the $L^1$ error. In contrast, the frequency average does not improve the convergence in a stable manner and we omit results obtained for this procedure.  

We also consider a \textit{combined averaging}, where we apply a mollifier average after having applied an angular average. Our experiments indicate that, for the best choice of $\sigma$, this combined method gives less error than the other two averaging procedures alone, but the extra improvement is not so pronounced. 
Note that for this combined averaging we also need to select a value for the parameter $\sigma$; we follow the same strategy and use the most convenient value for each example.

In Figures \ref{TS_phantom} to \ref{SL_phantom} we can see how the averaging procedures improve the convergence of the main term to the true potential. The examples considered include different geometries for the potential and a range of frequencies. Note that, even though the potentials considered do not belong to $H^{1/2}$, the main term of the standard formula is known to converge to the potential due to \cite{T17}.

\begin{figure}
\centering
\includegraphics[width=4.0cm]{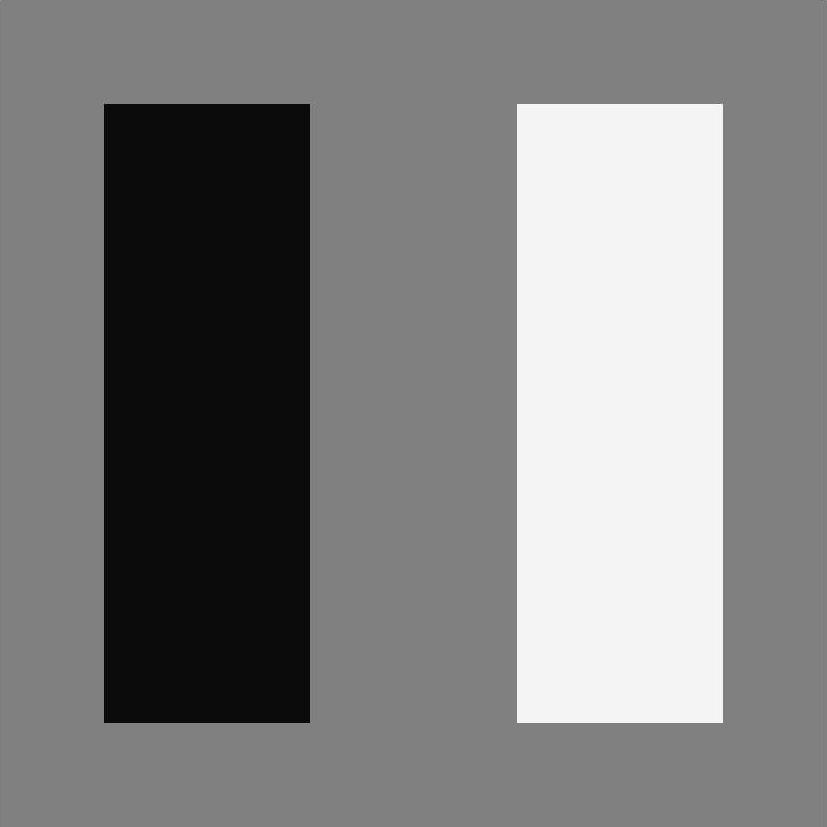} \hspace{0.1cm}
\includegraphics[width=4.0cm]{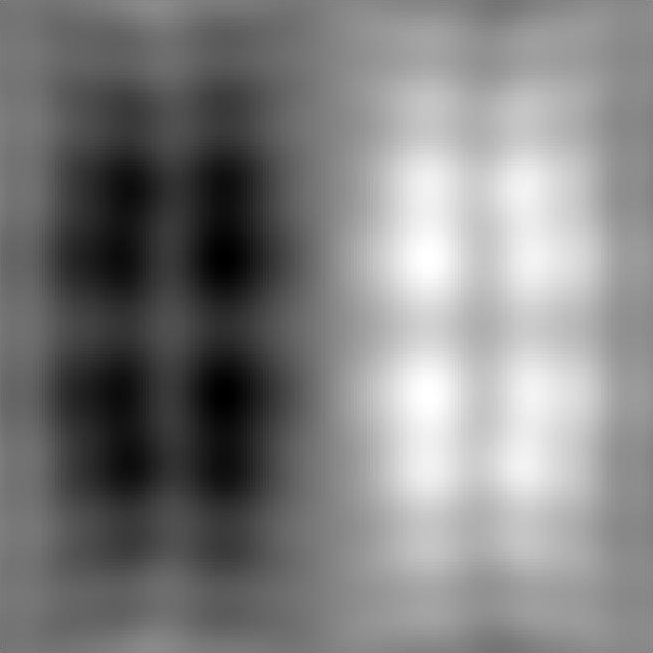} \\
\vspace{0.25cm}
\includegraphics[width=4.0cm]{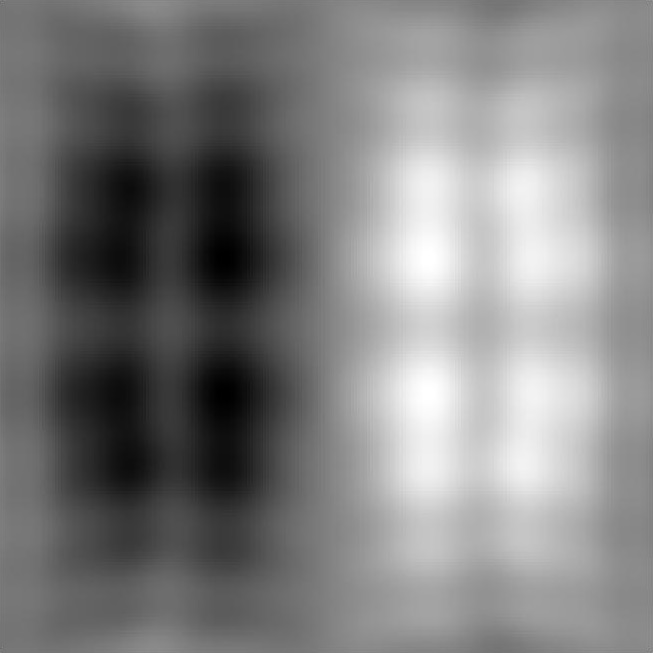} \hspace{0.1cm}
\includegraphics[width=4.0cm]{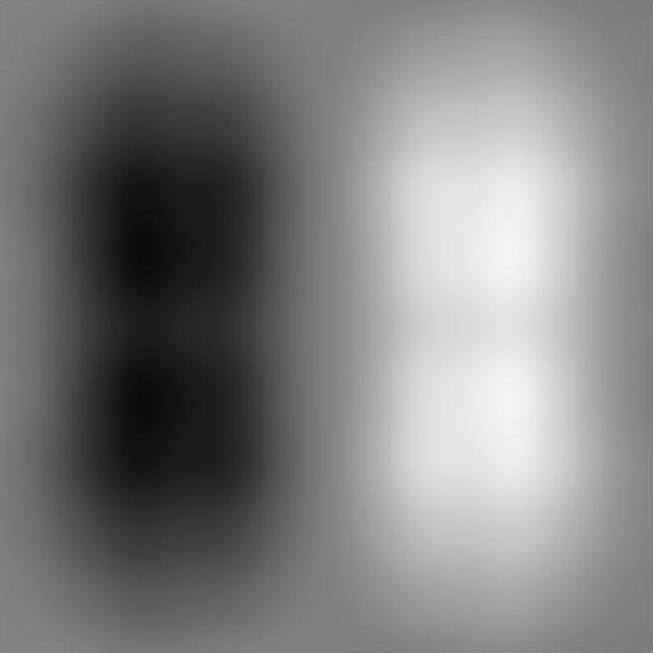} \hspace{0.1cm}
\includegraphics[width=4.0cm]{TS_k=10_angular.jpg} \\
\vspace{-0.1cm}
\caption{\small Rectangles $\lambda = 10$. Top: potential, standard main term. Bottom: mollifier, angular, combined.}
\label{TS_phantom}
\end{figure}

\begin{figure}
\centering
\includegraphics[width=4.0cm]{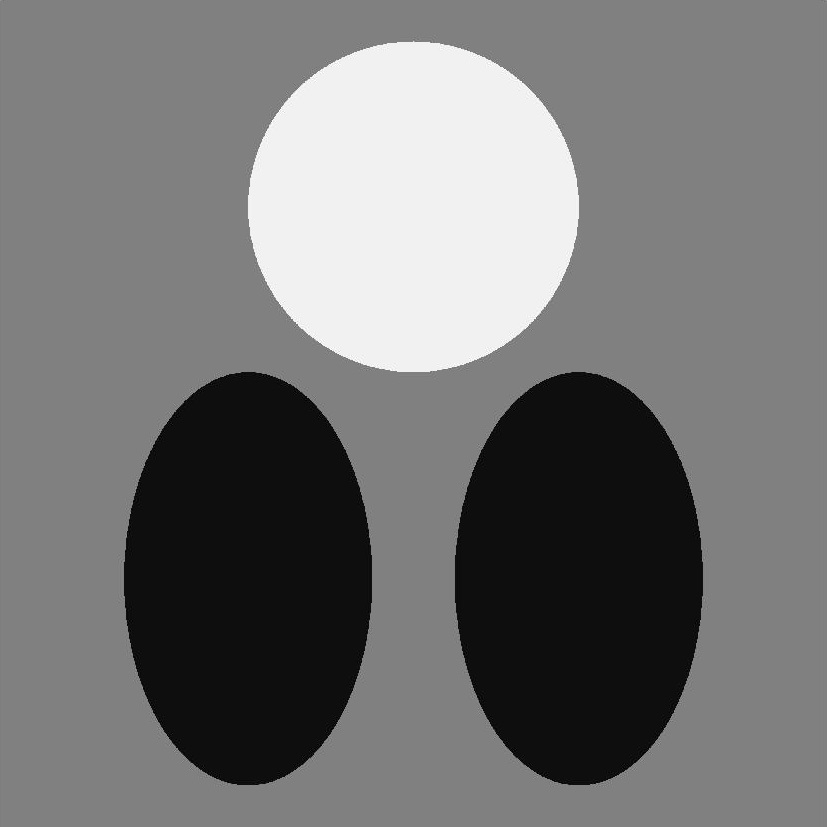} \hspace{0.1cm}
\includegraphics[width=4.0cm]{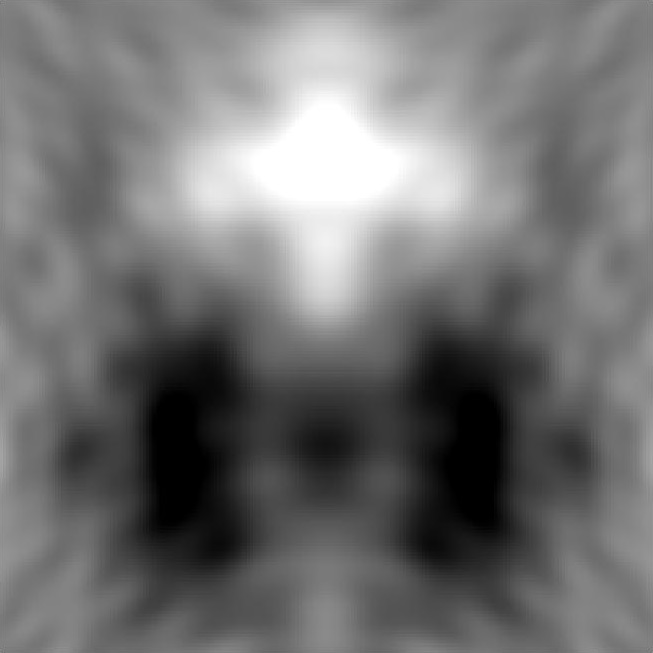} \\
\vspace{0.25cm}
\includegraphics[width=4.0cm]{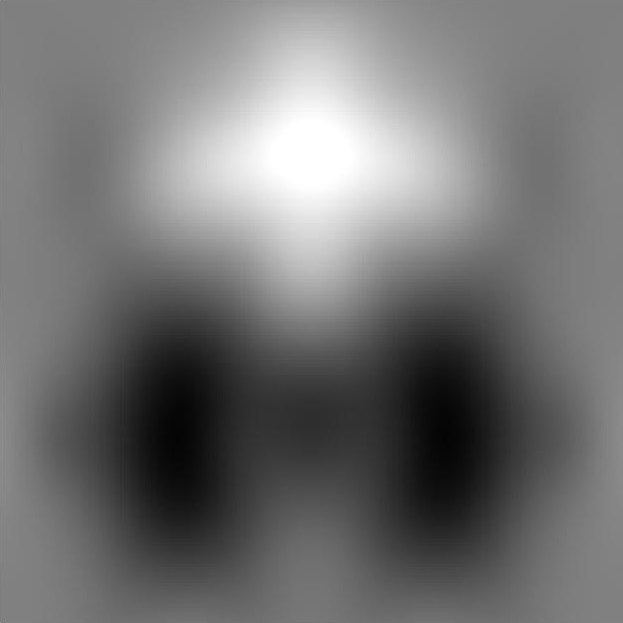} \hspace{0.1cm}
\includegraphics[width=4.0cm]{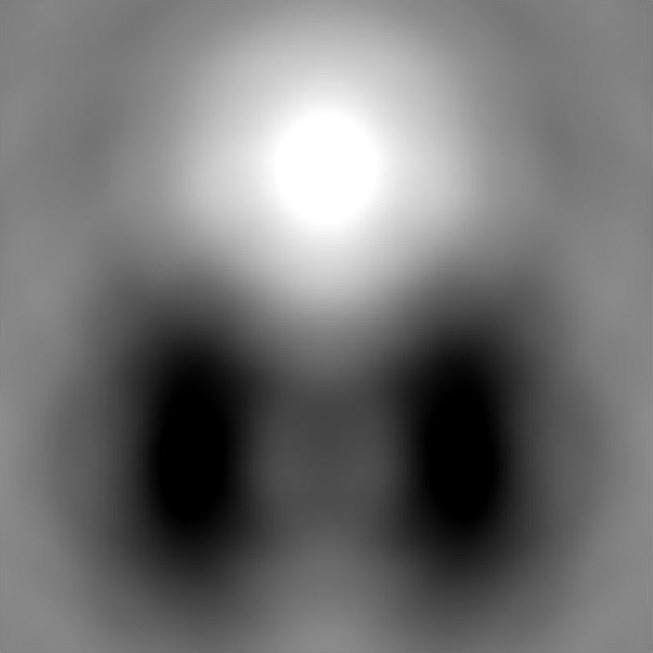}  \hspace{0.1cm}
\includegraphics[width=4.0cm]{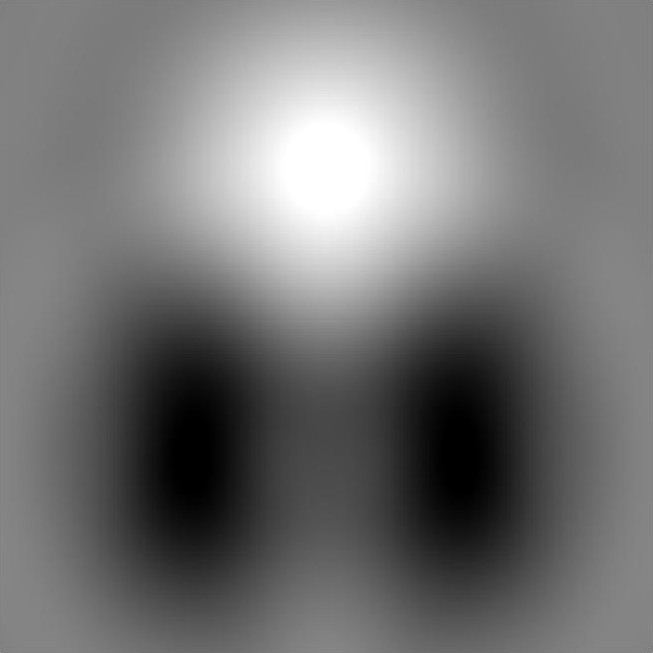} \\
\vspace{-0.1cm}
\caption{\small Ovals $\lambda = 15$. Top: potential, standard main term. Bottom: mollifier, angular, combined.}
\label{HL_phantom}
\end{figure}

\begin{figure}
\centering
\includegraphics[width=4.0cm]{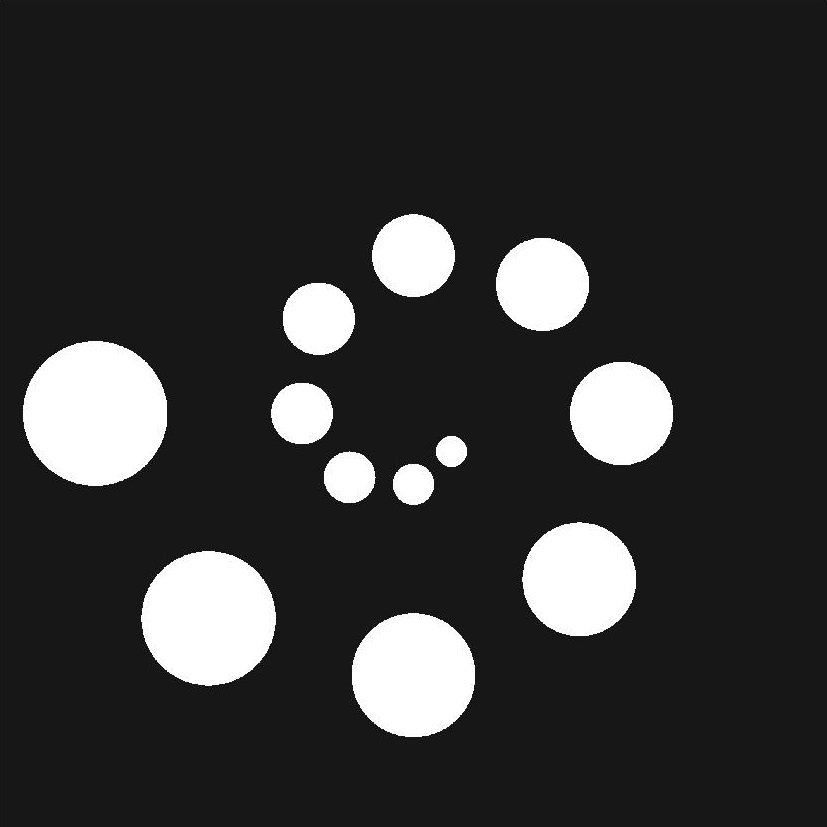} \hspace{0.1cm}
\includegraphics[width=4.0cm]{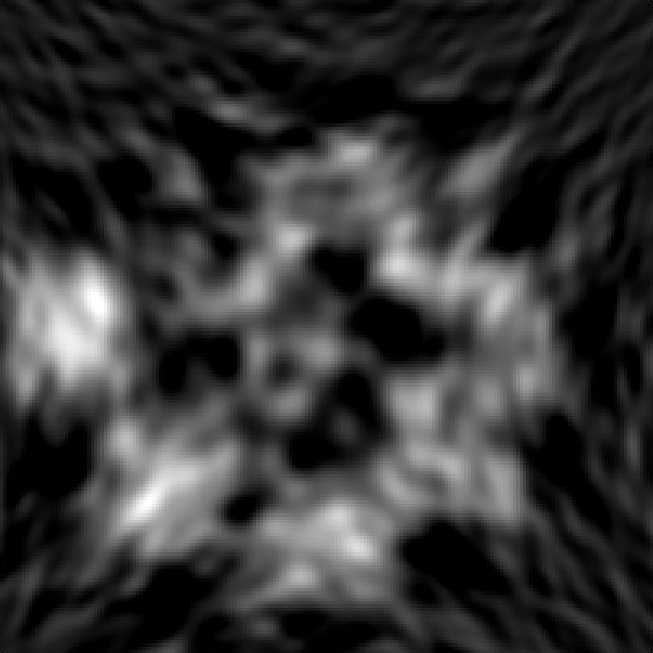} \\
\vspace{0.25cm}
\includegraphics[width=4.0cm]{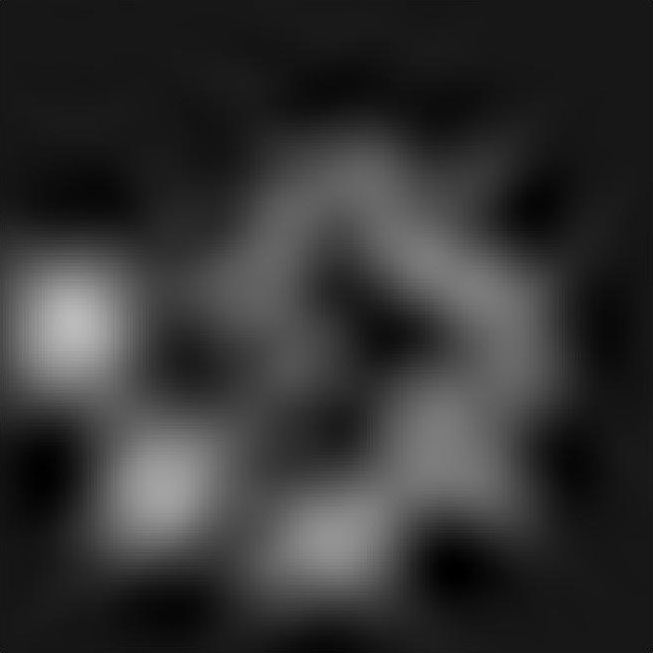} \hspace{0.1cm}
\includegraphics[width=4.0cm]{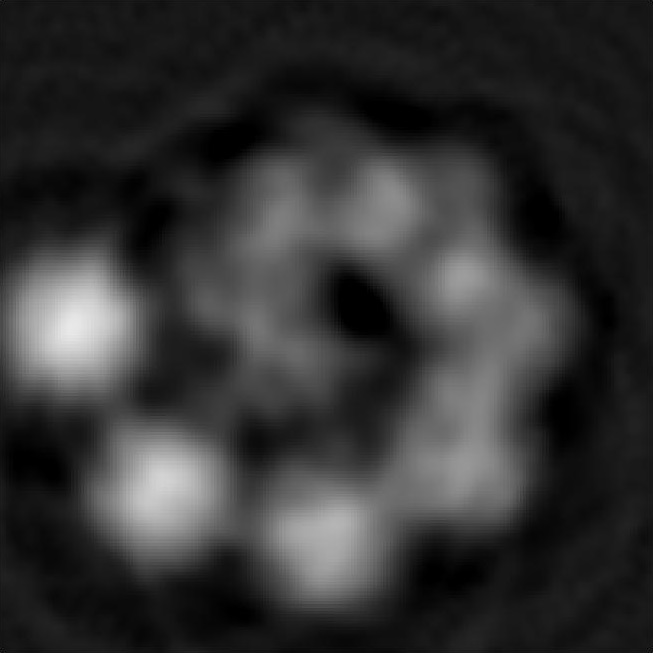}  \hspace{0.1cm}
\includegraphics[width=4.0cm]{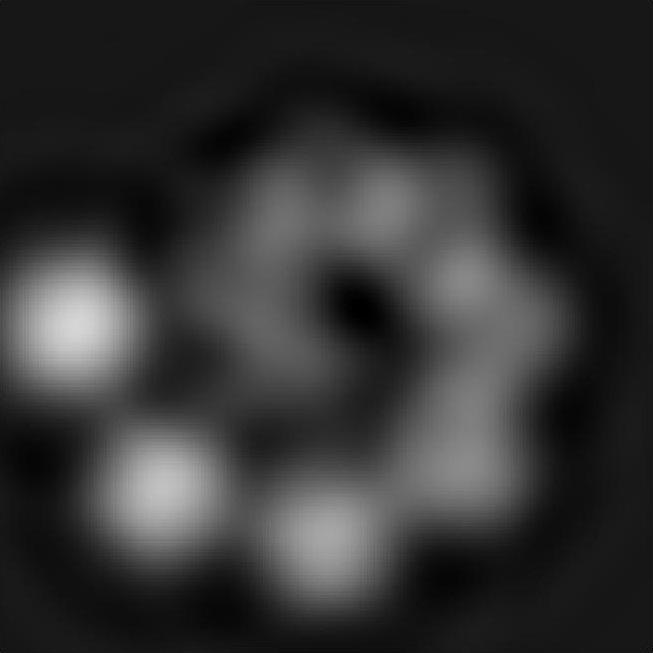} \\
\vspace{-0.1cm}
\caption{\small Circles spiral $\lambda = 30$. Top: potential, standard main term. Bottom: mollifier, angular, combined.}
\label{CI_phantom}
\end{figure}

\begin{figure}
\centering
\includegraphics[width=4.0cm]{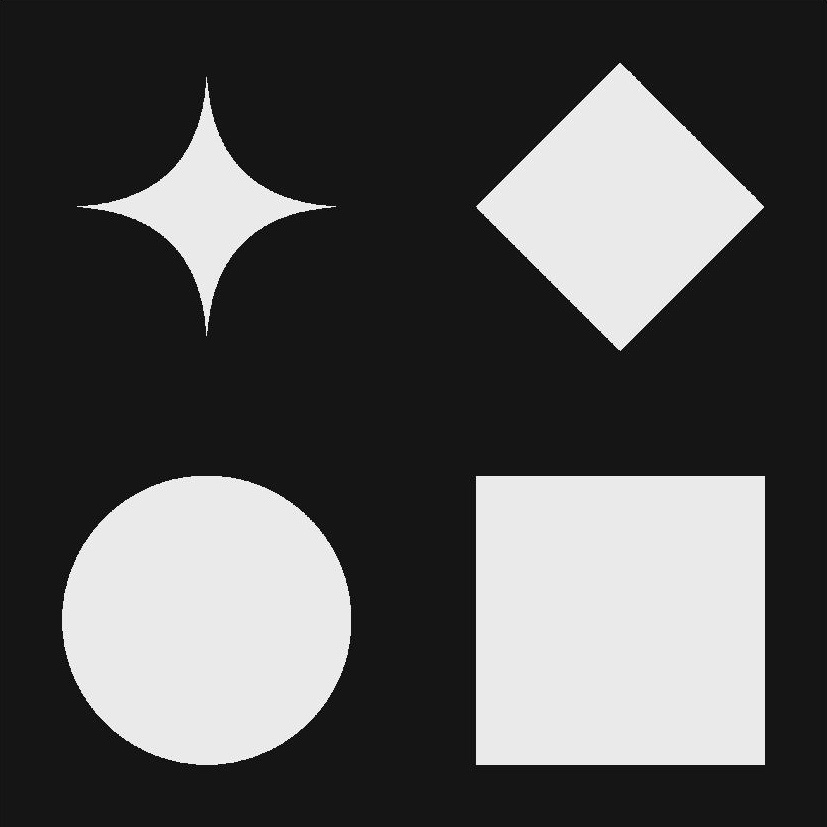} \hspace{0.1cm}
\includegraphics[width=4.0cm]{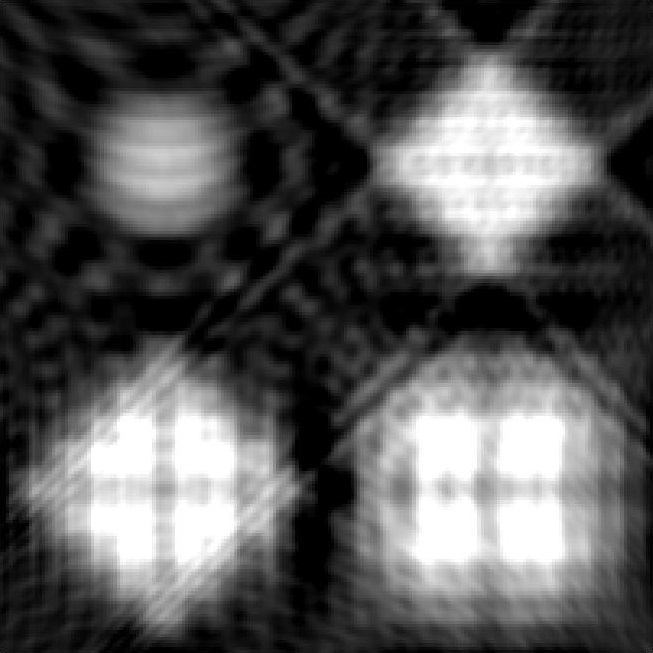} \\
\vspace{0.25cm}
\includegraphics[width=4.0cm]{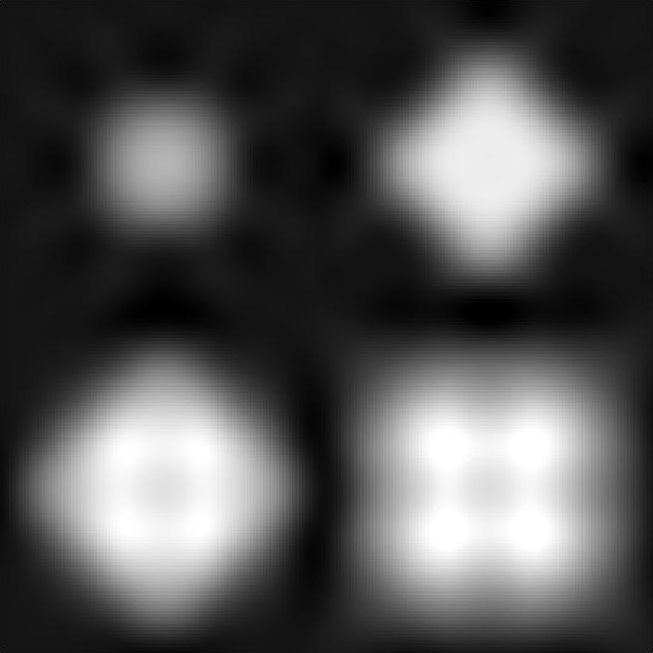} \hspace{0.1cm}
\includegraphics[width=4.0cm]{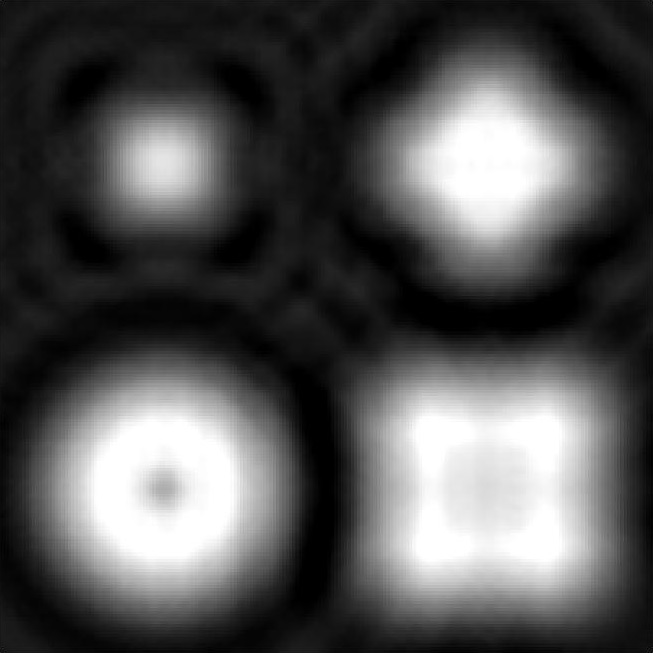}  \hspace{0.1cm}
\includegraphics[width=4.0cm]{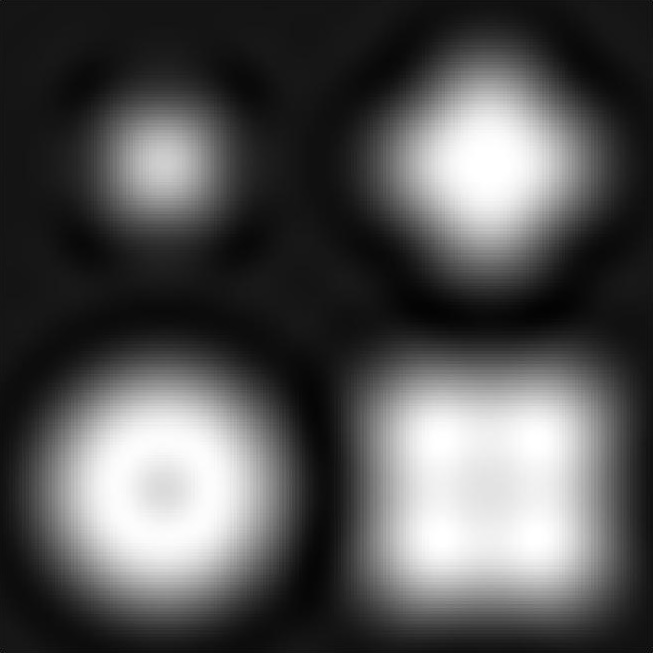} \\
\vspace{-0.1cm}
\caption{\small Geometric figures $\lambda = 50$. Top: potential, standard main term. Bottom: mollifier, angular, combined.}
\label{ME_phantom}
\vspace{12pt}
\end{figure}

\begin{figure}
\centering
\includegraphics[width=4.0cm]{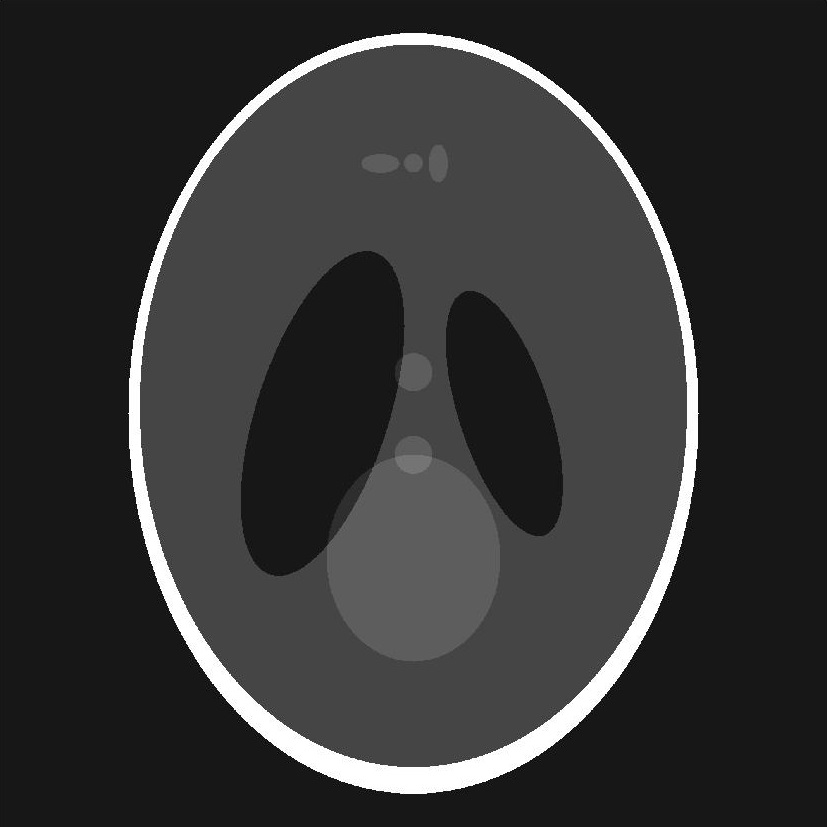} \hspace{0.1cm}
\includegraphics[width=4.0cm]{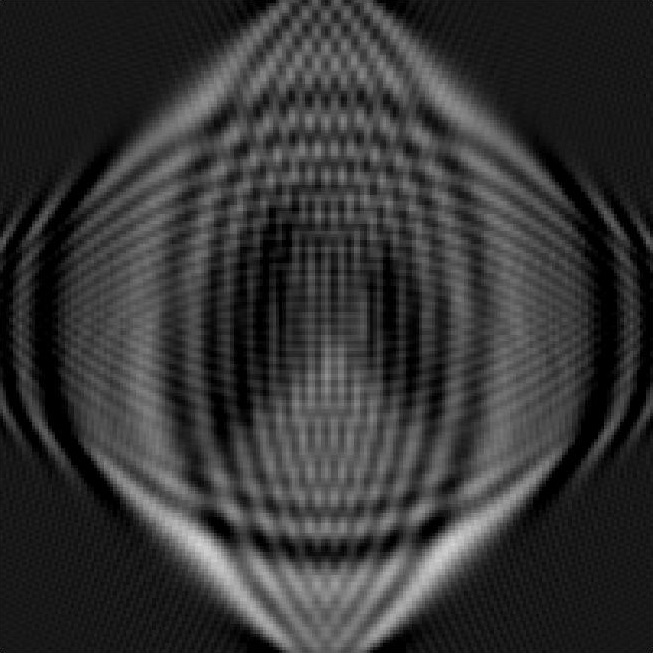} \\
\vspace{0.25cm}
\includegraphics[width=4.0cm]{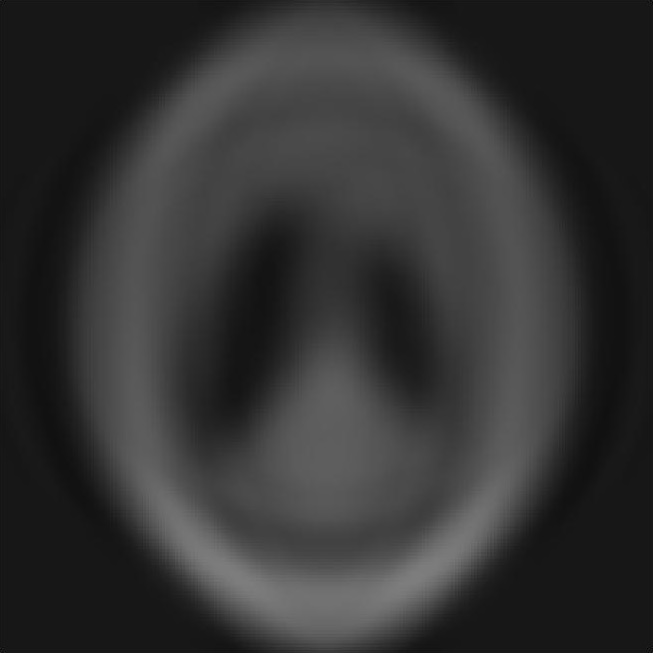} \hspace{0.1cm}
\includegraphics[width=4.0cm]{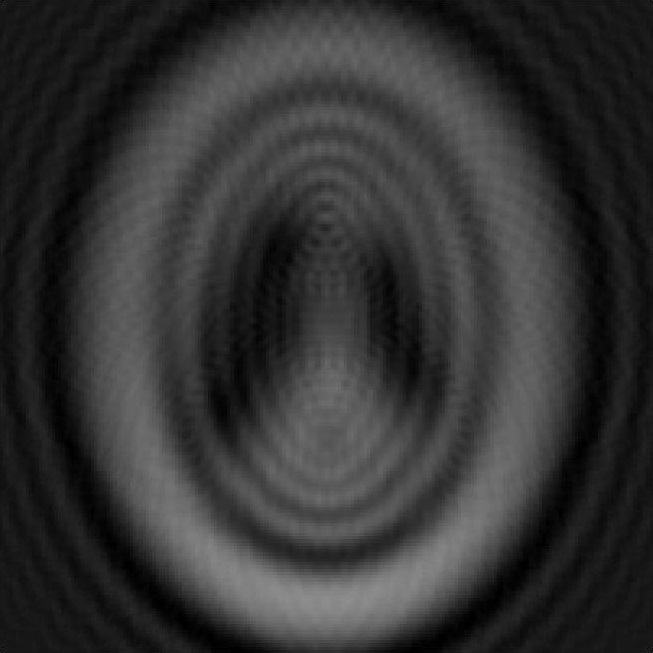}  \hspace{0.1cm}
\includegraphics[width=4.0cm]{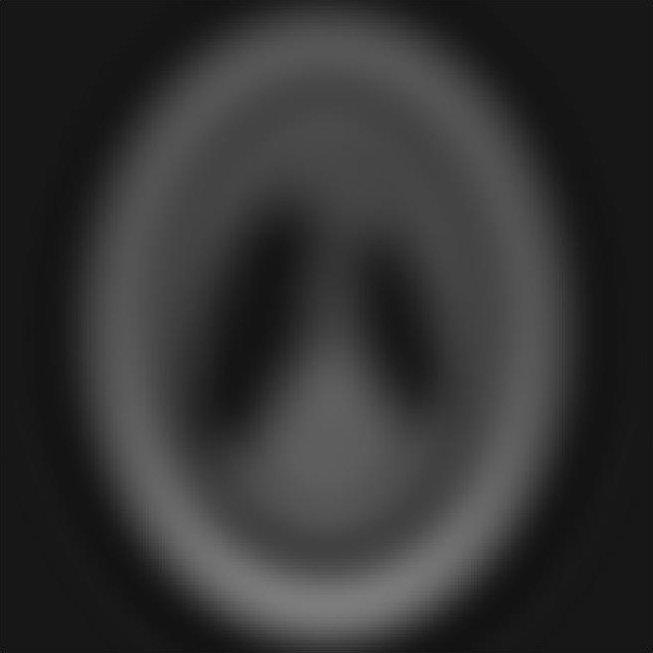} \\
\vspace{-0.1cm}
\caption{\small Shepp-Logan phantom $\lambda = 100$. Top: potential, standard main term. Bottom: mollifier, angular, combined.}
\label{SL_phantom}
\end{figure}

\begin{figure}
\renewcommand{\arraystretch}{1.15}
\centering
\begin{tabular}{|l|c|ccc|}
\hline
& $\lambda$ & Mollifier & Angular & Combined \\
\hline
Rectangles          & $10$  & $0\%$    & $17\%$ & $17\%$ \\
Rectangles          & $15$  & $1\%$    & $16\%$ & $16\%$ \\
Rectangles          & $20$  & $12\%$   & $22\%$ & $24\%$ \\
Rectangles          & $30$  & $18\%$   & $27\%$ & $28\%$ \\
\hline
Ovals               & $15$  & $9\%$    & $13\%$ & $15\%$ \\
Ovals               & $20$  & $13\%$   & $17\%$ & $18\%$ \\
Ovals               & $30$  & $12\%$   & $17\%$ & $19\%$ \\
\hline
Circles spiral      & $20$  & $23\%$   & $20\%$ & $26\%$ \\
Circles spiral      & $30$  & $20\%$   & $23\%$ & $24\%$ \\
Circles spiral      & $50$  & $21\%$   & $24\%$ & $26\%$ \\
\hline
Geometric figures   & $20$  & $7\%$    & $14\%$ & $14\%$ \\
Geometric figures   & $30$  & $11\%$   & $13\%$ & $15\%$ \\
Geometric figures   & $50$  & $18\%$   & $18\%$ & $21\%$ \\
\hline
Shepp-Logan phantom & $50$  & $30\%$   & $19\%$ & $30\%$ \\
Shepp-Logan phantom & $75$  & $31\%$   & $18\%$ & $32\%$ \\
Shepp-Logan phantom & $100$ & $33\%$   & $20\%$ & $34\%$ \\
\hline
\end{tabular}
\caption{\small Error reduction in the $L^1$ norm.}
\label{errorTable}
\end{figure}

In Figure \ref{errorTable} we can compare the error reductions in the $L^1$ norm for the mollifier average, the angular average and the combination of the angular and the mollifier average. The table contains the reduction for the examples in Figures \ref{TS_phantom} to \ref{SL_phantom} and for the same potentials with other choices of the frequency parameter $\lambda$; we see that results are stable in this sense. It is interesting that the effect of the averaging procedures seems to be more pronounced as the frequency increases. 
\\

\noindent\textbf{Acknowledgements:} The author is thankful to his PhD advisors Daniel Faraco and Keith M. Rogers for their comments and suggestions, and to Victor Arnaiz, \'Angel Castro and Javier Ramos for useful discussions. This work was partially supported by the ERC grant 301179 and the MINECO grants MTM2014-57769-1-P, SEV-2015-0554 and MTM2017-85934-P (Spain).

\end{document}